\documentclass{amsart}

\usepackage{amsmath}
\usepackage{amsfonts}
\usepackage{amssymb,enumerate}
\usepackage{amsthm}
\usepackage[all]{xy}
\usepackage{hyperref}

\newcommand{\Ker}{\operatorname{Ker}}

\newcommand{\Spec}{\operatorname{Spec}}

\newcommand{\td}{\operatorname{tr.deg.}}

\renewcommand{\dim}{\operatorname{dim}}

\newcommand{\Max}{\operatorname{Max}}

\newcommand{\h}{\operatorname{ht}}

\newcommand{\fn}{\frak{n}}
\newcommand{\fm}{\frak{m}}

\newcommand{\fq}{\frak{q}}
\newtheorem{thm}{Theorem}[section]
\newtheorem{cor}[thm]{Corollary}
\newtheorem{lem}[thm]{Lemma}
\newtheorem{prop}[thm]{Proposition}

\newtheorem{exam}[thm]{Example}
\newtheorem{rem}[thm]{Remark}

\begin{document}

\bibliographystyle{amsplain}

\date{}

\author{Parviz Sahandi}

\address{School of Mathematics, Institute for
Research in Fundamental Sciences (IPM), P.O. Box: 19395-5746, Tehran
Iran and Department of Mathematics, University of Tabriz, Tabriz,
Iran} \email{sahandi@ipm.ir}




\title[W-Jaffard domains in pullbacks]{W-Jaffard domains in pullbacks}

\begin{abstract} In this paper we study the class of $w$-Jaffard domains in pullback constructions,
and give new examples of these domains. In particular we give
examples to show that the two classes of $w$-Jaffard and Jaffard
domains are incomparable. As another application, we establish that
for each pair of positive integers $(n,m)$ with $n+1\leq m\leq
2n+1$, there is an (integrally closed) integral domain $R$ such that
$w$-$\dim(R)=n$ and $w[X]$-$\dim(R[X])=m$.
\end{abstract}

\maketitle

\section{Introduction}

Throughout this paper, $R$ denotes a (commutative integral) domain
with identity with quotient field $qf(R)$, and let $X$ be an
algebraically independent indeterminate over $R$. In \cite[Theorem
2]{S1} Seidenberg proved that if $R$ has finite Krull dimension,
then
$$\dim(R)+1\leq\dim(R[X])\leq2(\dim(R))+1.$$
Moreover, Krull \cite{Kr} has shown that if $R$ is any
finite-dimensional Noetherian ring, then $\dim(R[X])=1+\dim(R)$ (cf.
also \cite[Theorem 9]{S1}). Seidenberg subsequently proved the same
equality in case $R$ is any finite-dimensional Pr\"{u}fer domain. To
unify and extend such results on Krull-dimension, Jaffard
\cite{Jaf1} introduced and studied the \emph{valuative dimension}
denoted by $\dim_v(R)$, for a domain $R$. This is the maximum of the
ranks of the valuation overrings of $R$. Jaffard proved in
\cite[Chapitre IV]{Jaf1} that, if $R$ has finite valuative
dimension, then $\dim_v(R[X])=1+\dim_v(R)$, and that if $R$ is a
Noetherian or a Pr\"{u}fer domain, then $\dim(R)=\dim_v(R)$. In
\cite{ABDFK} Anderson, Bouvier, Dobbs, Fontana and Kabbaj introduced
the notion of \emph{Jaffard domains}, as finite dimensional integral
domains $R$ such that $\dim(R)=\dim_v(R)$, and studied this class of
domain systematically (see also \cite{BK}).

The $v$, $t$ and $w$-operations in integral domains are of special
importance in multiplicative ideal theory and was investigated by
many authors in the 1980's. Ideal $w$-multiplication converts ring
notions such as Dedekind, Noetherian, Pr\"{u}fer, and
quasi-Pr\"{u}fer, respectively to Krull, strong Mori, P$v$MD, and
UM$t$. As the $w$-counterpart of Jaffard domains, in \cite{S}, we
introduced the class of \emph{$w$-Jaffard domains}, as integral
domains $R$ such that $w\text{-}\dim(R)=w\text{-}\dim_v(R)<\infty$.
In this paper we study the transfer of $w$-Jaffard domains in
pullback constructions, in order to provide original examples.

We need to recall some notions from star operations. Let $F(R)$
denotes the set of nonzero fractional ideals, and $f(R)$ be the set
of all nonzero finitely generated fractional ideals of $R$. Let $*$
be a star operation on the domain $R$. For every $A\in F(R)$, put
$A^{*_f}:=\bigcup F^*$, where the union is taken over all $F\in
f(R)$ with $F\subseteq A$. It is easy to see that $*_f$ is a star
operation on $R$. A star operation $*$ is said to be \emph{of finite
character} if $*=*_f$. We say that a nonzero ideal $I$ of $R$ is a
\emph{$*$-ideal}, if $I^*=I$; a \emph{$*$-prime}, if $I$ is a prime
$*$-ideal of $R$; a \emph{$*$-maximal}, if $I$ is maximal in the set
of $*$-prime ideals of $R$. The set of $*$-maximal ideals of $R$ is
denoted by $*$-$\Max(R)$. It has become standard to say that a star
operation $*$ is \emph{stable} if $(A\cap B)^{*}=A^*\cap B^*$ for
all $A$, $B\in F(R)$.

Given a star operation $*$ on an integral domain $R$, it is possible
to construct a star operation $\widetilde{*}$, which is stable and
of finite character defined as follows: for each $A\in F(R)$,
$$
A^{\widetilde{*}}:=\{x\in qf(R)|xJ\subseteq A,\text{ for some
}J\subseteq R, J\in f(R), J^*=R\}.
$$

The most widely studied star operations on $R$ have been the
identity $d$, $v$, $t:=v_f$, and $w:=\widetilde{v}$ operations,
where $A^{v}:=(A^{-1})^{-1}$, with $A^{-1}:=(R:A):=\{x\in
qf(R)|xA\subseteq R\}$. In this work we mostly deal with the
$w$-operation.

It is well-known that $t$-$\Max(R)=w$-$\Max(R)$, every $t$-prime
ideal is a $w$-prime ideal, and that every prime subideal of a
prime $w$-ideal of $R$ is also a $w$-ideal.

Let $*$ be a star operation on a domain $R$. The $*$-{\it Krull
dimension of} $R$ is defined as
$$
*\text{-}\dim(R):=\sup\{n| P_1\subset\cdots\subset P_n\text{
where } P_i \text{ is }*\text{-prime}\}.
$$
If the set of $*$-prime ideals is an empty set then pose
$*\text{-}\dim(R)=0$. Note that, the notions of
$\widetilde{*}$-dimension, $t$-dimension, and of $w$-dimension
have received a considerable interest by several authors (cf. for
instance, \cite{S, Sah, Sbull, H, HM, F1, F2}).

Now we recall a special case of a general construction for semistar
operations (see \cite{S}). Let $X$, $Y$ be two indeterminates over
$R$, and let $K:=qf(R)$. Set $R_1:=R[X]$, $K_1:=K(X)$ and take the
following subset of $\Spec(R_1)$:
$$\Theta_1^w:=\{Q_1\in\Spec(R_1)|\text{ }Q_1\cap R=(0)\text{ or }(Q_1\cap R)^w\subsetneq R\}.$$
Set $\mathfrak{S}_1^w:=R_1[Y]\backslash(\bigcup\{Q_1[Y]
|Q_1\in\Theta_1^w\})$ and:
$$E^{\circlearrowleft_{\mathfrak{S}_1^w}}:=E[Y]_{\mathfrak{S}_1^w}\cap
K_1, \text{   for all }E\in F(R_1).$$

It is proved in \cite[Theorem 2.1]{S} that, the mapping
$w[X]:=\circlearrowleft_{\mathfrak{S}_1^w}: F(R_1)\to F(R_1)$,
$E\mapsto E^{w[X]}$ is a stable star operation of finite character
on $R[X]$, i.e., $\widetilde{w[X]}=w[X]$. If $X_1,\cdots,X_r$ are
indeterminates over $R$, for $r\geq2$, we let
$$
w[X_1,\cdots,X_r]:=(w[X_1,\cdots,X_{r-1}])[X_r].
$$
For an integer $r$, put $w[r]$ to denote $w[X_1,\cdots,X_r]$, and
$R[r]$ to denote $R[X_1,\cdots,X_r]$.

\begin{prop}(\cite[Theorem 3.1]{S})\label{finite} For each positive integer $r$ and for $n:=w\text{-}\dim(R)$ we have
$$r+n\leq w[r]\text{-}\dim(R[r])\leq r+(r+1)n.$$
\end{prop}

\begin{prop}(\cite[Lemma 4.4]{Sbull})\label{loc} Let $R$ be an integral domain
and $n$ be an integer. Then
$$
w[n]\text{-}\dim(R[n])=\sup\{\dim(R_M[n])|M\in w\text{-}\Max(R)\}.
$$
\end{prop}

A valuation overring $V$ of $R$, is called a \emph{$w$-valuation
overring of $R$}, provided $F^w\subseteq FV$, for each $F\in f(R)$.
Following \cite{S}, the \emph{$w$-valuative dimension} of $R$ is
defined as:
$$
w\text{-}\dim_v(R):=\sup\{\dim(V)|V\text{ is }w\text{-valuation
overring of }R\}.
$$

\begin{prop}(\cite[Lemma 2.5]{Sbull})\label{v} For each domain $R$,
$$
w\text{-}\dim_v(R)=\sup\{\dim_v(R_P)|P\in w\text{-}\Max(R)\}.
$$
\end{prop}

\begin{prop}(\cite[Theorem 4.2]{Sbull})\label{wdim} Let $R$ be an
integral domain, and $n$ be a positive integer. Then the following
statements are equivalent:
\begin{itemize}
\item[(1)] $w$-$\dim_v(R)=n$.

\item[(2)] $w[n]\text{-}\dim(R[n])=2n$.

\item[(3)] $w[r]\text{-}\dim(R[r])=r+n$ for all $r\geq n-1$.
\end{itemize}
\end{prop}

It is observed in \cite{S} that $w$-$\dim(R)\leq w$-$\dim_v(R)$.
We say that $R$ is a \emph{$w$-Jaffard domain}, if
$w\text{-}\dim(R)=w\text{-}\dim_v(R)<\infty$. It is proved in
\cite{S}, that $R$ is a $w$-Jaffard domain if and only if
$$w[n]\text{-}\dim(R[n])=w\text{-}\dim(R)+n,$$
for each positive integer $n$.

Recall that an integral domain is called a \emph{strong Mori domain}
if it satisfies the ascending chain condition on $w$-ideals (cf.
\cite{WM}). Also recall that an integral domain $R$ is called a
\emph{UM$t$-domain}, if every upper to zero in $R[X]$ is a maximal
$t$-ideal \cite[Section 3]{HZ}. It is shown in \cite[Corollary 4.6
and Theorem 4.14]{S} that a strong Mori domain or a UM$t$ domain of
finite $w$-dimension is a $w$-Jaffard domain. In particular every
Krull domain is a $w$-Jaffard domain (of $w$-dimension 1).

If $F\subseteq K$ are fields, then $\td(K/F)$ stands for the
\emph{transcendence degree} of $K$ over $F$. Let $T$ be an integral
domain, $M$ a maximal ideal of $T$, $k=T/M$ and $\varphi:T\to k$ the
canonical surjection. Let $D$ be a proper subring of $k$ and
$R=\varphi^{-1}(D)$ be the pullback of the following diagram:
\begin{displaymath}
\xymatrix{ R \ar[r] \ar[d] &
D \ar[d] \\
T \ar[r]^{\varphi} & k. }
\end{displaymath}
In Section 2 we prove that if $F:=qf(D)$ then the followings hold:
\begin{itemize}
\item[(1)] $w\text{-}\dim(R)=\max\{w\text{-}\dim(T),w\text{-}\dim(D)+\dim(T_M)\}$.

\item[(2)] $w\text{-}\dim_v(R)=\max\{w\text{-}\dim_v(T),w\text{-}\dim_v(D)+\dim_v(T_M)+\td(k/F)\}$.

\item[(3)] If $T$ is quasilocal, then $R$ is a $w$-Jaffard domain if and only if $D$ is a $w$-Jaffard
domain, $T$ is a Jaffard domain, and $k$ is algebraic over $F$.
\end{itemize}

Using these results, in Section 3 we give examples to show that the
two classes of $w$-Jaffard and Jaffard domains are incomparable, and
an example of a $w$-Jaffard domain which is not a strong Mori nor a
UM$t$ domain. Also we observed that a Mori domain need not be a
$w$-Jaffard domain. As another application in Section 4 we prove
that for any pair of positive integers $(n,m)$ with $n+1\leq m\leq
2n+1$, there is an integrally closed integral domain $R$ such that
$w$-$\dim(R)=n$ and $w[X]$-$\dim(R[X])=m$, which is similar to a
result of Seidenberg \cite{S2}.

For the convenience of the reader, the following displays a diagram
of implications summarizing the relations between the main classes
of integral domains involved in this work.

\begin{center}
\setlength{\unitlength}{6pt}
\begin{picture}(40,33)(0,0)
\put(11,15.6){\circle*{.4}} \put(11,25.6){\circle*{.4}}
\put(11,25.6){\vector(0,-1){10}} \put(16.6,31){Dedekind}
\put(17.7,21){Krull} \put(9,26.6){Pr\"{u}fer} \put(9,16.6){P$v$MD}
\put(16.5,1){$w$-Jaffard} \put(17.5,11){Jaffard}
\put(-10,20){Quasi-Pr\"{u}fer} \put(41,20){Noetherian}
\put(40,0){\circle*{.4}} \put(40,10){\vector(0,-1){10}}
\put(41,0){Mori} \put(-4,10){UM$t$} \put(41,10){Strong Mori}
\put(20,0){\circle*{.4}} \put(20,10){\circle*{.4}}
\put(0,10){\circle*{.4}} \put(0,20){\circle*{.4}}
\put(40,10){\circle*{.4}} \put(40,20){\circle*{.4}}
\put(20,20){\circle*{.4}} \put(20,30){\circle*{.4}}
\put(20,30){\vector(0,-1){10}} \put(40,20){\vector(0,-1){10}}
\put(0,20){\vector(0,-1){10}} \put(20,30){\vector(2,-1){20}}
\put(20,20){\vector(2,-1){20}} \put(0,10){\vector(2,-1){20}}
\put(0,20){\vector(2,-1){20}} \put(20,30){\vector(-2,-1){9}}
\put(11,25.5){\vector(-2,-1){11}} \put(20,20){\vector(-2,-1){9}}
\put(11,15.5){\vector(-2,-1){11}} \put(40,10){\vector(-2,-1){20}}
\put(40,20){\vector(-2,-1){20}}
\end{picture}\\
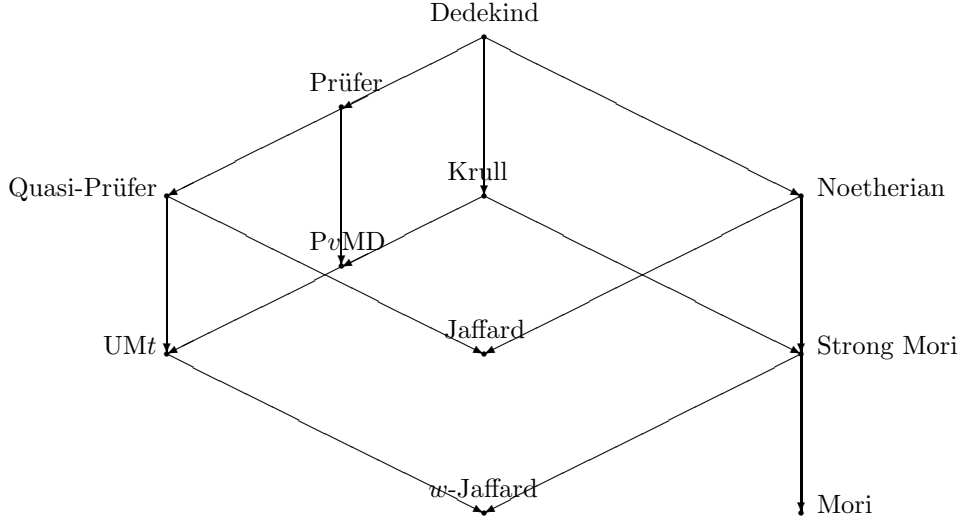
\vspace{.5cm} A ring-theoretic perspective for $w$-Jaffard property.
\end{center}

\section{Pullbacks}

It is shown in \cite[Theorem 4.14]{S} that, a UM$t$ domain of finite
$w$-dimension is a $w$-Jaffard domain. Now we give an example of a
$w$-Jaffard non UM$t$ domain. Recall that recently Houston and
Mimouni in \cite[Theorem 4.2]{HM} proved that, if $m, n$ are
integers with $1\leq m\leq n$, and $B\subseteq\{2,\cdots, n\}$ with
$|B|=n-m$, then there exists a local Noetherian domain $R$ such that
$\dim(R)=n$, $t$-$\dim(R)=m$, and for each $i\in B$, every prime
ideal of height $i$ is a non-$t$-prime. Now let $n=3$, $m=2$ and
$B=\{3\}$. Then there exists a local Noetherian domain $(R,\fm)$
such that $\dim(R)=3=\h(\fm)$, $t$-$\dim(R)=2$, and that $\fm$ is a
non-$t$-prime. Consequently we have $w$-$\dim(R)=2$. Since $R$ is
Noetherian thus it is strong Mori and hence is a $w$-Jaffard domain.
But $R$ is not a UM$t$ domain since $w$-$\dim(R)=2$ (cf.
\cite[Theorem 3.7]{HZ}). In Example \ref{nonSM} we will give a
$w$-Jaffard domain which is not a strong Mori nor a UM$t$ domain.

To avoid unnecessary repetition, let us fix the notation. Let $T$
be an integral domain, $M$ a maximal ideal of $T$, $k=T/M$ and
$\varphi:T\to k$ the canonical surjection. Let $D$ be a proper
subring of $k$ and $R=\varphi^{-1}(D)$ be the pullback of the
following diagram:
\begin{displaymath}
\xymatrix{ R \ar[r] \ar[d] &
D \ar[d] \\
T \ar[r]^{\varphi} & k. }
\end{displaymath}
We assume that $R\subsetneq T$, and we refer to this diagram as a
diagram of type $(\square)$ and if the quotient field of $D$ is
equal to $k$, we refer to the diagram as a diagram of type
$(\square^*)$. The case where $T=V$ is a valuation domain of the
form $K+M$, where $K$ is a field and $M$ is the maximal ideal of
$V$ is of crucial interest, known as classical ``$D+M$''
construction.

Recall that $(R:T)=M$ is a prime ideal of $R$ and therefore $M$
is a divisorial ideal (or a $v$-ideal) of $R$. Thus $M$ is a
$w$-prime ideal of $R$. Also recall that $R/M\simeq D$, and $R$
and $T$ have the same quotient field. Moreover, $T$ is quasilocal
if and only if every ideal of $R$ is comparable (under inclusion)
to $M$. For each prime ideal $P$ of $R$ with $P\nsupseteq M$,
there is a unique prime ideal $Q$ of $T$ with $Q\cap R=P$ and
such that $R_P=T_Q$. For more details on general pullbacks, we
refer the reader to \cite{Fon, GH, GH1}, and \cite{BG} for
classical $D+M$ constructions.

\begin{lem}\label{w} For a
diagram of type $(\square)$, suppose that $P$ is a prime ideal of
$D$ and $Q$ is a prime ideal of $R$ such that
$Q=\varphi^{-1}(P)$. Then $P$ is a $w$-prime (resp. $w$-maximal)
ideal of $D$ if and only if $Q$ is a $w$-prime (resp.
$w$-maximal) ideal of $R$
\end{lem}

\begin{proof} By \cite[Lemma
3.1]{Mim} we have $Q^w=\varphi^{-1}(P^w)$. So that if $P$ is a
$w$-prime ideal of $D$ then $Q$ is $w$-prime ideal of $R$.
Conversely if $Q$ is a $w$-prime ideal of $R$, then we have
$\varphi^{-1}(P^w)=\varphi^{-1}(P)$. Let $a\in P^w$. Then
$\varphi^{-1}(a)\subseteq\varphi^{-1}(P^w)=\varphi^{-1}(P)$. So
that $a\in P$ since $\varphi^{-1}(a)\neq\emptyset$. Thus $P^w=P$.
The other assertion is clear.
\end{proof}

It is well-known that \cite[Proposition 2.1(5)]{Fon} for a diagram
of type $(\square)$, if $T$ is quasilocal, we have
$\dim(R)=\dim(D)+\dim(T)$. The following proposition gives a
satisfactory analogue of this equality.

\begin{prop}\label{wdim=dim} For a
diagram of type $(\square)$, assume that $T$ is quasilocal. Then
$$
w\text{-}\dim(R)=w\text{-}\dim(D)+\dim(T).
$$
\end{prop}

\begin{proof} Let $n:=w\text{-}\dim(R)$, $s:=w\text{-}\dim(D)$, and
$t:=\dim(T)$. Suppose that $P_1\subset\cdots\subset P_s$ be a chain
of $w$-prime ideals of $D$. Let $Q_i:=\varphi^{-1}(P_i)$ which is a
$w$-prime ideal of $R$ by Lemma \ref{w}. Thus $M\subset
Q_1\subset\cdots\subset Q_s$. Also consider a chain
$L_1\subset\cdots\subset L_t=M$ of prime ideals of $T$. Note that
each $L_j$ is a $w$-prime ideal of $R$. Now we have a chain
$L_1\subset\cdots\subset L_t=M\subset Q_1\subset\cdots\subset Q_s$
of distinct $w$-prime ideals of $R$. This means that $s+t\leq n$.
Conversely suppose that $L_1\subset\cdots\subset L_r=M\subset
Q_1\subset\cdots\subset Q_u$ be a chain of distinct $w$-prime ideals
of $R$ such that $r+u=n$. Thus $L_1\subset\cdots\subset L_r=M$ is a
chain of prime ideals of $T$ and hence $r\leq t$. On the other hand
by setting $P_i:=Q_i/M$, we have a chain $P_1\subset\cdots\subset
P_u$ of $w$-prime ideals of $D$ by Lemma \ref{w}, and hence $u\leq
s$. Therefore $n=r+u\leq t+s$ completing the proof.
\end{proof}

\begin{rem} For a
diagram of type $(\square)$, assume that $T$ is quasilocal and $D=F$
is a field. Then by \cite[Theorem 3.1(2)]{Mim1}, we have $R$ is a
DW-domain (that is the $d$- and $w$-operations are the same). Hence
$w\text{-}\dim(R)=\dim(R)$. So that the equality in Lemma
\ref{wdim=dim} would be $(\dim(R)=)w\text{-}\dim(R)=\dim(T)$.
\end{rem}

The following proposition is inspired by \cite[Proposition
2.3]{ABDFK}.

\begin{prop}\label{main} For a
diagram of type $(\square^*)$, assume that $T$ is quasilocal.
Then the followings hold:
\begin{itemize}
\item[(a)]
$w[r]\text{-}\dim(R[r])=w[r]\text{-}\dim(D[r])+\dim(T[r])-\dim(k[r])$
for each positive integer $r$,

\item[(b)] $w\text{-}\dim_v(R)=w\text{-}\dim_v(D)+\dim_v(T)$,

\item[(c)] $R$ is a $w$-Jaffard domain $\Leftrightarrow$ $D$ is a $w$-Jaffard
domain and $T$ is a Jaffard domain.
\end{itemize}
\end{prop}

\begin{proof} (a) By Proposition \ref{wdim=dim} we have
$w$-$\dim(R)<\infty$ if and only if $w$-$\dim(D)<\infty$ and
$\dim(T)<\infty$. Hence $w[r]\text{-}\dim(R[r])<\infty$ if and only
if $w[r]\text{-}\dim(D[r])$ and $\dim(T[r])$ are finite numbers by
Proposition \ref{finite}. Thus we can assume that each domain is
finite ($w$-)dimensional.

Using \cite[Lemma 4.4 and Corollary 4.5]{Sbull}, there is a
$w$-maximal ideal $\fq$ of $R$ and a $w[r]$-maximal ideal $Q$ of
$R[r]$ such that $\fq=\varphi^{-1}(Q)$, and
$w[r]\text{-}\dim(R[r])=\h(Q)=r+\h(\fq[r])=\dim(R_\fq[r])$. Note
that since $M$ is a $w$-prime ideal of $R$ we have $M\subseteq\fq$.
Thus by Lemma \ref{w} we have $P:=\fq/M$ is a $w$-maximal ideal of
$D$. Next we claim that $\sup\{\dim(D_L[r])|L\in
w\text{-}\Max(D)\}=\dim(D_{P}[r])$, then Proposition \ref{loc} will
implies that $w[r]\text{-}\dim(D[r])=\dim(D_{P}[r])$. Let $L\in
w\text{-}\Max(D)$ and set $\fq_0:=\varphi^{-1}(L)$. We have the
following diagrams:
\begin{displaymath}
\xymatrix{ R_{\fq_0} \ar[r] \ar[d] & D_L \ar[d] & R_{\fq} \ar[r] \ar[d] & D_P \ar[d]\\
T \ar[r]^{\varphi} & k, & T \ar[r]^{\varphi} & k. }
\end{displaymath}
Therefore by \cite[Proposition 2.3]{ABDFK} we have
$\dim(D_L[r])+\dim(T[r])-r=\dim(R_{\fq_0}[r])\leq\dim(R_{\fq}[r])=\dim(D_P[r])+\dim(T[r])-r$,
where the inequality holds by Proposition \ref{loc}. Thus
$\dim(D_L[r])\leq\dim(D_P[r])$ for each $L\in w\text{-}\Max(D)$, and
hence $\sup\{\dim(D_L[r])|L\in w\text{-}\Max(D)\}=\dim(D_{P}[r])$.
Therefore we have
$$
w[r]\text{-}\dim(R[r])=w[r]\text{-}\dim(D[r])+\dim(T[r])-\dim(k[r]).
$$

(b) First suppose that $w\text{-}\dim_v(R)<\infty$. Then
$w\text{-}\dim(D)+\dim(T)=w\text{-}\dim(R)<\infty$, and so both
$w\text{-}\dim(D)$ and $\dim(T)$ are finite. In addition we claim
that $\dim_v(T)<\infty$. To this end let $(V,N)$ be a valuation
overring of $T$ and set $P:=N\cap T$. So that $P\subseteq M$ and
thus $P$ is a prime ideal of $R$. Hence $P$ is in fact a $w$-prime
ideal of $R$. Since $R_P\subseteq T_{R\backslash P}\subseteq
T_P\subseteq V$ we obtain that $V$ is a $w$-valuation overring of
$R$ by \cite[Theorem 3.9]{FL}, and consequently $\dim(V)\leq
w\text{-}\dim_v(R)$. This means that $\dim_v(T)\leq
w\text{-}\dim_v(R)<\infty$. Next we observe that
$w\text{-}\dim_v(D)<\infty$. Using Proposition \ref{v}, there exists
a $w$-prime ideal $P$ of $D$ such that
$w\text{-}\dim_v(D)=\dim_v(D_P)$. Let $Q:=\varphi^{-1}(P)$, which is
a $w$-prime ideal of $R$. Note that we have $M\subseteq Q$ and thus
$(R\backslash Q)\cap M=\emptyset$, and $\varphi(R\backslash
Q)=D\backslash P$. Therefore by \cite[Proposition 1.9]{Fon} we have
the following pullback diagram:
\begin{displaymath}
\xymatrix{ R_Q \ar[r] \ar[d] &
D_P \ar[d] \\
T \ar[r]^{\varphi} & k. }
\end{displaymath}
If $B$ is an $n$-dimensional overring of $D_P$, then
$A:=\varphi^{-1}(B)$ is an overring of $R_Q$, and \cite[Proposition
2.1(5)]{Fon} yields that $n+\dim(T)=\dim(A)$. Thus
$n+\dim(T)\leq\dim_v(R_Q)$. This means that
$w\text{-}\dim_v(D)=\dim_v(D_P)\leq \dim_v(R_Q)\leq
w\text{-}\dim_v(R)<\infty$, where the second inequality holds by
Proposition \ref{v}.

Let $r$ be a positive integer such that
$r\geq\max\{w\text{-}\dim_v(R),
w\text{-}\dim_v(D),\dim_v(T)\}-1$. Then by Proposition \ref{wdim}
and \cite[Theorem 6]{A} we have
\begin{align*}
w[r]\text{-}\dim(R[r])= & w\text{-}\dim_v(R)+r, \\[1ex]
w[r]\text{-}\dim(D[r])= & w\text{-}\dim_v(D)+r, \\[1ex]
\dim(T[r])= & \dim_v(T)+r.
\end{align*}
Then by (a),
$w\text{-}\dim_v(R)+r=(w\text{-}\dim_v(D)+r)+(\dim_v(T)+r)-r$,
yielding (b) in case $w\text{-}\dim_v(R)<\infty$.

To complete the proof of (b) we show that
$w\text{-}\dim_v(R)<\infty$ whenever $w\text{-}\dim_v(D)$ and
$\dim_v(T)$ are both finite. Let $r$ be a positive integer such
that $$r\geq\max\{w\text{-}\dim_v(D), \dim_v(T)\}-1.$$ Then by
(a), Proposition \ref{wdim}, and \cite[Theorem 6]{A} we have
$w[r]\text{-}\dim(R[r])=w[r]\text{-}\dim(D[r])+\dim(T[r])-r=(w\text{-}\dim_v(D)+r)+
(\dim_v(T)+r)-r=w\text{-}\dim_v(D)+\dim_v(T)+r$. Hence
$w\text{-}\dim_v(R)<\infty$ by another appeal to Proposition
\ref{wdim}.

(c) Since $w\text{-}\dim(R)=w\text{-}\dim(D)+\dim(T)$ and
$w$-$\dim(B)\leq w$-$\dim_v(B)$ and $\dim(B)\leq \dim_v(B)$ for a
domain $B$, (c) follows directly from (b).
\end{proof}

Recall from \cite{BSH} the notion of CPI (complete pre-image)
extension of a domain $R$ with respect to a prime ideal $P$ of $R$;
this is denoted $R(P)$ and is defined by the following pullback
diagram:
\begin{displaymath}
\xymatrix{ R(P) \ar[r] \ar[d] &
R/P \ar[d] \\
R_P \ar[r]^{\varphi} & R_P/PR_P. }
\end{displaymath}
Here $\varphi$ is the canonical homomorphism.

\begin{cor} Let $R$ be an integral domain, and let $P$ be a prime of $R$.
Then the CPI-extension $R(P)$ is a $w$-Jaffard domain
$\Leftrightarrow$ $R/P$ is a $w$-Jaffard domain and $R_P$ is a
Jaffard domain.
\end{cor}

In \cite[Theorem 2.6]{ABDFK}, Anderson, Bouvier, Dobbs, Fontana
and Kabbaj proved that for a diagram of type $(\square)$ such
that $T$ is quasilocal and $F:=qf(D)$ then
$\dim_v(R)=\dim_v(D)+\dim_v(T)+\td(k/F)$ and hence $R$ is a
Jaffard domain if and only if $D$ and $T$ are Jaffard domains and
$k$ is algebraic over $F$. Now we have:

\begin{thm}\label{main1} For a
diagram of type $(\square)$, assume that $T$ is quasilocal and let
$F=qf(D)$. Then
\begin{itemize}
\item[(a)]
$w\text{-}\dim_v(R)=w\text{-}\dim_v(D)+\dim_v(T)+\td(k/F)$,

\item[(b)] $R$ is a $w$-Jaffard domain $\Leftrightarrow$ $D$ is a $w$-Jaffard
domain, $T$ is a Jaffard domain, and $k$ is algebraic over $F$.
\end{itemize}
\end{thm}

\begin{proof} (a) Split the pullback diagram $(\square)$ into two parts:
\begin{displaymath}
\xymatrix{ R \ar[r] \ar[d] &
D \ar[d] \\
S:=\varphi^{-1}(D) \ar[r] \ar[d] &
F \ar[d] \\
T \ar[r]^{\varphi} & k. }
\end{displaymath}
Now the upper diagram is of type $(\square^*)$, and $S$ is
quasilocal. Thus by Proposition \ref{main}(b) we have
$w\text{-}\dim_v(R)=w\text{-}\dim_v(D)+\dim_v(S)$. Also from the
lower diagram, \cite[Proposition 2.5]{ABDFK} yields that
$\dim_v(S)=\dim_v(T)+\td(k/F)$. We thus have the desired equality.

(b) Since $w\text{-}\dim(R)=w\text{-}\dim(D)+\dim(T)$, (b) follows
directly from (a).
\end{proof}

In Example \ref{K} we will give an example of a $w$-Jaffard
domain which is not Jaffard. Using Theorem \ref{main1} together
with \cite[Theorem 2.6]{ABDFK} we have the following corollary.

\begin{cor} For a
diagram of type $(\square)$, assume that $T$ is quasilocal and let
$F=qf(D)$. Then $R$ is a $w$-Jaffard domain which is not Jaffard
$\Leftrightarrow$ $D$ is a $w$-Jaffard domain which is not Jaffard,
$T$ is a Jaffard domain and $k$ is algebraic over $F$.
\end{cor}

We pause here to give some concrete applications of the above
theory to the classical $D+M$ constructions.

\begin{cor} Let $V$ be a nontrivial valuation domain of the form
$V=K+M$, where $K$ is a field and $M$ is the maximal ideal of $V$.
Let $R=D+M$, where $D$ is a proper subring of $K$ and let $F=qf(D)$.
Then
\begin{itemize}
\item[(a)] $w\text{-}\dim_v(R)=w\text{-}\dim_v(D)+\dim(V)+\td(K/F)$,

\item[(b)] $R$ is a $w$-Jaffard domain $\Leftrightarrow$ $D$ is a $w$-Jaffard
domain, $V$ is finite-dimensional, and $K$ is algebraic over $F$.
\end{itemize}
\end{cor}

A ``global'' type of $D+M$ constructions arise from $T=K[[X]]$,
the formal power series ring over a field $K$, by considering
$M=XT$ and a subring $D$ of $K$.

\begin{cor} Let $K$ be a field, $D$ a subring of $K$ with quotient field
$F$, $R=D+XK[[X]]$. Then
\begin{itemize}
\item[(a)] $w\text{-}\dim(R)=w\text{-}\dim(D)+1$,

\item[(b)] $w\text{-}\dim_v(R)=w\text{-}\dim_v(D)+\td(K/F)+1$.

\item[(c)] $R$ is a $w$-Jaffard domain $\Leftrightarrow$ $D$ is a $w$-Jaffard
domain and $K$ is algebraic over $F$.
\end{itemize}
\end{cor}

We next proceed to generalize the previous ``quasilocal'' theory. In
this direction we prove the ``global'' analogue of Propositions
\ref{wdim=dim} and \ref{main}(b). Before that we need two lemmas.

\begin{lem}\label{lem*} For a
diagram of type $(\square^*)$ we have:
\begin{itemize}
\item[(a)]
$w\text{-}\dim(R)=\max\{w\text{-}\dim(T),w\text{-}\dim(D)+\dim(T_M)\}$,

\item[(b)] $w\text{-}\dim_v(R)=\max\{w\text{-}\dim_v(T),w\text{-}\dim_v(D)+\dim_v(T_M)\}$.
\end{itemize}
\end{lem}

\begin{proof} (a) We have
$w\text{-}\dim(R)=\sup\{\dim(R_P)|P\in w\text{-}\Max(R)\}$. Now let
$P\in w\text{-}\Max(R)$ such that $w\text{-}\dim(R)=\dim(R_P)$. If
$P\not\supset M$ then $R_P=T_Q$ for some $Q\in\Spec(T)$ such that
$P=Q\cap R$. Thus $Q$ and $M$ are incomparable prime ideals of $T$.
Hence using \cite[Lemma 3.3]{FGH}, we see that $Q$ is a $w$-maximal
ideal of $T$. On the other hand if $P\supseteq M$, then
$\dim(R_P)=\dim(D_Q)+\dim(T_M)$ for some $Q\in w\text{-}\Max(R)$
such that $P=\varphi^{-1}(Q)$. Then we have the inequality $\leq$ in
(a). We have two cases to consider:

$1^\circ$ If $P\not\supset M$ then $R_P=T_Q$ for some $Q\in
w\text{-}\Max(R)$ such that $P=Q\cap R$. We claim that
$w$-$\dim(T)=\dim(T_Q)$. Suppose the contrary that there exists
$L\in w\text{-}\Max(R)$ such that $w$-$\dim(T)=\dim(T_L)$ and
$\dim(T_Q)\lneq\dim(T_L)$. Set $P_1:=L\cap R$. Consequently
$R_{P_1}=T_L$ by \cite[Proposition 1.11]{GH1}. Thus
$w\text{-}\dim(R)=\dim(R_P)=\dim(T_Q)\lneq\dim(T_L)=\dim(R_{P_1})$.
This implies that $P_1$ is not a $w$-ideal of $R$ contradicting
\cite[Theorem 2.6(2)]{GH1}. Thus in this case we have the equality
in (a).

$2^\circ$ If $P\supseteq M$, then $\dim(R_P)=\dim(D_Q)+\dim(T_M)$
for some $Q\in w\text{-}\Max(D)$ such that $P=\varphi^{-1}(Q)$ (by
Proposition \ref{wdim=dim}). We claim that
$w$-$\dim(D)=\dim(D_Q)$. Suppose the contrary that there exists
$L\in w\text{-}\Max(D)$ such that $w$-$\dim(D)=\dim(D_L)$ and
$\dim(D_Q)\lneq\dim(D_L)$. Set $P_1:=\varphi^{-1}(L)$.
Consequently $\dim(R_{P_1})=\dim(D_L)+\dim(T_M)$. Thus
$w\text{-}\dim(R)=\dim(R_P)=\dim(D_Q)+\dim(T_M)\lneq\dim(D_L)+\dim(T_M)=\dim(R_{P_1})$.
This implies that $P_1$ is not a $w$-ideal of $R$ contradicting
Lemma \ref{w}. Thus in this case again we have the equality in
(a).

(b) We have $w\text{-}\dim_v(R)=\sup\{\dim_v(R_P)|P\in
w\text{-}\Max(R)\}$ by Proposition \ref{v}. The rest of the proof
is the same as part (a).
\end{proof}

\begin{lem}\label{lemf} For a
diagram of type $(\square)$ assume that $D=F$ is a field and let
$d=\td(k/F)$. Then
\begin{itemize}
\item[(a)] $w\text{-}\dim(R)=\max\{w\text{-}\dim(T), \dim(T_M)\}$,

\item[(b)] $w\text{-}\dim_v(R)=\max\{w\text{-}\dim_v(T), \dim_v(T_M)+d\}$.
\end{itemize}
\end{lem}

\begin{proof} (a) Note that $M$ is a $w$-prime ideal of $R$. Then we have
$w\text{-}\dim(R)=\max\{\sup\{\dim(R_P)|P\in w\text{-}\Max(R),\text{
and }P\not\supset M \},\dim(R_M)\}$. Like Lemma \ref{lem*} the
inequality $\leq$ holds. Let $P\in w\text{-}\Max(R)$ such that
$w\text{-}\dim(R)=\dim(R_P)$. If $P=M$, then we have
$\dim(R_P)=\dim(T_M)$ by \cite[Proposition 2.1(5)]{Fon}. If not we
have $P\not\supset M$. Then $R_P=T_Q$ for some $Q\in\Spec(T)$ such
that $P=Q\cap R$. Using \cite[Lemma 3.3]{FGH}, we see that $Q$ is a
$w$-maximal ideal of $T$. We claim that $w$-$\dim(T)=\dim(T_Q)$.
Suppose the contrary that there exists $L\in w\text{-}\Max(R)$ such
that $w$-$\dim(T)=\dim(T_L)$ and $\dim(T_Q)\lneq\dim(T_L)$. Set
$P_1:=L\cap R$. If $P_1\not\supset M$ then $R_{P_1}=T_L$. Thus
$w\text{-}\dim(R)=\dim(R_P)=\dim(T_Q)\lneq\dim(T_L)=\dim(R_{P_1})$.
This implies that $P_1$ is not a $w$-ideal. But if $L\subseteq M$
then $P_1\subseteq M$ and hence $P_1$ is a $w$-prime ideal which is
a contradiction. So that $L\not\subset M$. Thus $P_1$ is a $w$-prime
ideal by \cite[Theorem 2.6(2)]{GH1} which is again a contradiction.
Therefore $w\text{-}\dim(R)=\dim(R_P)=\dim(T_Q)=w\text{-}\dim(T)$.

(b) It is the same as part (a) noting that we have
$$
w\text{-}\dim_v(R)=\max\{\sup\{\dim_v(R_P)|P\in w\text{-}\Max(R),\text{
and }P\not\supset M \},\dim_v(R_M)\},
$$ by Proposition \ref{v} and using \cite[Theorem
2.11(b)]{ABDFK} instead of \cite[Proposition 2.1(5)]{Fon}.
\end{proof}

By combining Lemmas \ref{lem*} and \ref{lemf} we have:

\begin{thm}\label{main2} For a
diagram of type $(\square)$, let $F=qf(D)$ and $d:=\td(k/F)$.
Then:
\begin{itemize}
\item[(a)] $w\text{-}\dim(R)=\max\{w\text{-}\dim(T),w\text{-}\dim(D)+\dim(T_M)\}$.

\item[(b)] $w\text{-}\dim_v(R)=\max\{w\text{-}\dim_v(T),w\text{-}\dim_v(D)+\dim_v(T_M)+d\}$.
\end{itemize}
\end{thm}

An integral domain $R$ is said to be a \emph{$w$-locally Jaffard
domain} if $R_P$ is a Jaffard domain for each $w$-prime ideal $P$
of $R$. It is easy to see that a $w$-locally Jaffard domain of
finite $w$-valuative dimension is a $w$-Jaffard domain. Now we
have the following corollary which is $w$-analogue of
\cite[Corollary 2.12]{ABDFK}.

\begin{cor}\label{main3} For a
diagram of type $(\square)$, let $F$ be the quotient field of $D$.
Then:
\begin{itemize}
\item[(a)] $R$ is a $w$-locally Jaffard domain $\Leftrightarrow$ $D$ and $T$ are $w$-locally Jaffard
domains, and $k$ is algebraic over $F$.

\item[(b)] If $T$ is a $w$-locally Jaffard domain with $w$-$\dim_v(T)<\infty$, $D$ is a $w$-Jaffard
domain, and $k$ is algebraic over $F$, then $R$ is a $w$-Jaffard
domain.
\end{itemize}
\end{cor}

A ``global'' type of $D+M$ constructions arise from $T=K[X]$, the
polynomial ring over a field $K$, by considering $M=XT$ and a
subring $D$ of $K$. In this case neither $T$ nor $R$ is quasilocal.
Theorem \ref{main2} yields:

\begin{cor}\label{g} Let $K$ be a field, $D$ a subring of $K$ with quotient field
$F$, $R=D+XK[X]$ and $d=\td(K/F)$. Then:
\begin{itemize}
\item[(a)] $w\text{-}\dim(R)=w\text{-}\dim(D)+1$.

\item[(b)] $w\text{-}\dim_v(R)=w\text{-}\dim_v(D)+d+1$.

\item[(c)] $R$ is a $w$-Jaffard domain $\Leftrightarrow$ $D$ is a $w$-Jaffard
domain and $K$ is algebraic over $F$.
\end{itemize}
\end{cor}

\section{Examples}

It is well known that \cite[Theorem 6.7.8]{FHP} a finite
dimensional domain $R$ has Pr\"{u}fer integral closure if and
only if each overring of $R$ is a Jaffard domain. Similarly in
\cite{Sq} we showed that a finite $w$-dimensional domain $R$ has
Pr\"{u}fer integral closure if and only if each overring of $R$
is a $w$-Jaffard domain. Thus in particular each overring of a
finite dimensional domain is Jaffard if and only if each overring
is $w$-Jaffard. In the following two examples we show that the
classes of $w$-Jaffard and Jaffard domains are incomparable.

The next example gives a positive answer to our question in
\cite[page 238]{S}, which asked that ``is it possible to find a
$w$-Jaffard non-Jaffard domain?'' There is an old question (see
\cite{BK}) asking if is it possible to find a UFD (or a Krull
domain) which is not Jaffard. We note that if there exists a Krull
domain which is not Jaffard, then we do have an example of a
$w$-Jaffard domain which is not Jaffard. But to the best of
author's knowledge there is not such an example.

\begin{exam}\label{K} For each $n\geq 3$ there is an integral domain $R_n$
which is $w$-Jaffard of $w$-$\dim(R_n)=n$ but not a Jaffard
domain.

\noindent Let $K$ be a field and let $W,X,Y,Z$ be indeterminates
over $K$. Put $L=K(W,X,Y,Z)$. Now, $V_1=K(W,X,Z)+M_1$, where
$M_1=YK(W,X,Z)[Y]_{(Y)}$, is a (discrete) rank 1 valuation domain of
$L$ with maximal ideal $M_1$. Let $(V,M)$ be a rank 1 valuation
domain of the form $V=K(W,X,Y)+M$, where $M=ZK(W,X,Y)[Z]_{(Z)}$.
With $\tau$ denoting the canonical surjection $V\to K(W,X,Y)$,
consider the pullback $V'=\tau^{-1}(K(W,X)[Y]_{(Y+1)})=K(W,X)+M'$
where $M'=(Y+1)K(W,X)[Y]_{(Y+1)}+ZK(W,X,Y)[Z]_{(Z)}$. Thus
$\dim(V')=2$. Finally with $\psi$ denoting the canonical surjection
$V'\to K(W,X)$, consider the pullback
$V_2=\psi^{-1}(K(W)[X]_{(X)})=K(W)+M_2$, where
$M_2=XK(W)[X]_{(X)}+(Y+1)K(W,X)[Y]_{(Y+1)}+ZK(W,X,Y)[Z]_{(Z)}$, is a
valuation domain of $L$ with maximal ideal $M_2$ and we have
$\dim(V_2)=3$. Further, $V_1$ and $V_2$ are incomparable. If not, it
would follow from the one-dimensionality of $V_1$ that $V_2\subset
V_1$. Then we would have $V_1=(V_2)_M$, whence
$YK(W,X,Z)[Y]_{(Y)}=M_1=M(V_2)_M=M$ and $1=YY^{-1}\in MV=M$, a
contradiction. Thus $V_1$ and $V_2$ are incomparable. Then
$T:=V_1\cap V_2$ is a three dimensional Pr\"{u}fer domain with
$\fm_1:=M_1\cap T$ and $\fm_2:=M_1\cap T$ as maximal ideals such
that $T_{\fm_1}=V_1$ and $T_{\fm_2}=V_2$ by \cite[Theorem
11.11]{Na}. With $\varphi:T\to T/\fm_1(\cong V_1/M_1\cong K(W,X,Z))$
denoting the canonical surjection, consider the pullback
$R_3:=\varphi^{-1}(K[W,X])$. Notice that $T$ is a DW-domain since it
is a Pr\"{u}fer domain, $d:=\td(K(W,X,Z)/K(W,X))=1$, and that
$K[W,X]$ is a Noetherian Krull domain. In particular $K[W,X]$ is a
Jaffard domain (of dimension 2) and $w$-Jaffard domain (of
$w$-dimension 1). Thus using Theorem \ref{main2} we have:
\begin{align*}
w\text{-}\dim(R_3)=&\max\{w\text{-}\dim(T),w\text{-}\dim(K[W,X])+\dim(T_{\fm_1})\} \\[1ex]
                = & \max\{3,1+1\}=3, \text{ and}\\[1ex]
w\text{-}\dim_v(R_3)=&\max\{w\text{-}\dim_v(T),w\text{-}\dim_v(K[W,X])+\dim_v(T_{\fm_1})+d\} \\[1ex]
   = & \max\{3,1+1+1\}=3.
\end{align*}
This means that $R_3$ is a $w$-Jaffard domain of $w$-dimension 3.
But by \cite[Theorem 2.11]{ABDFK} we have
\begin{align*}
\dim(R_3)=&\max\{\dim(T),\dim(K[W,X])+\dim(T_{\fm_1})\} \\[1ex]
       = & \max\{3,2+1\}=3, \text{ and}\\[1ex]
\dim_v(R_3)=&\max\{\dim_v(T),\dim_v(K[W,X])+\dim_v(T_{\fm_1})+d\} \\[1ex]
         = & \max\{3,2+1+1\}=4.
\end{align*}
Therefore $R_3$ is not a Jaffard domain.

Now set $F:=qf(R_3)$. Suppose that $V:=F+M$ is a rank 1 valuation
domain with maximal ideal $M$. Set $R_4:=R_3+M$. It is easy to
see that $R_4$ is $w$-Jaffard of $w$-$\dim(R_4)=4$,
$\dim(R_4)=4$, and $\dim_v(R_4)=5$. Iterating in the same way we
obtain $R_n$ with desired properties.
\end{exam}

\begin{exam}\label{Ko} For each $n\geq 2$ there is an integral domain $R_n$
which is Jaffard of $\dim(R_n)=n$ but not a $w$-Jaffard domain.

\noindent Let $K$ be a field and let $X,Y,Z$ be indeterminates over
$K$. Let $C:=K[X,Y,Z]$ and set $P:=(X)$ and $Q:=(Y,Z)$. Let $T:=C_S$
where $S:=C\backslash (P\cup Q)$, which is a multiplicatively closed
subset of $C$. Then $\Max(T)=\{PT,QT\}$, $\dim(T_{PT})=1$ and
$\dim(T)=\dim(T_{QT})=2$. Next notice that we have a surjective ring
homomorphism $\psi:C_P\to K(Y,Z)$ sending $f/g\mapsto
f(0,Y,Z)/g(0,Y,Z)$, with $\Ker(\psi)=PC_P$. Thus we have $T/PT\cong
C_P/PC_P\cong K(Y,Z)$. With $\varphi:T\to T/PT$ denoting the
canonical surjection, consider the pullback
$R_2:=\varphi^{-1}(K(Y))$. Note that $T$ is a Noetherian Krull
domain. Thus $T$ is a 2 dimensional Jaffard domain and a $w$-Jaffard
domain of $w$-dimension 1. Also notice that $d:=\td(K(Y,Z)/K(Y))=1$.
Thus by \cite[Theorem 2.11]{ABDFK} we have
\begin{align*}
\dim(R_2)=&\max\{\dim(T),\dim(K(Y))+\dim(T_{PT})\} \\[1ex]
       =&\max\{2,0+1\}=2, \text{ and}\\[1ex]
\dim_v(R_2)=&\max\{\dim_v(T),\dim_v(K(Y))+\dim_v(T_{PT})+d\} \\[1ex]
         = & \max\{2,0+1+1\}=2.
\end{align*}
Therefore $R_2$ is a Jaffard domain of dimension 2. On the other
hand using Theorem \ref{main2} we have:
\begin{align*}
w\text{-}\dim(R_2)=&\max\{w\text{-}\dim(T),w\text{-}\dim(K(Y))+\dim(T_{PT})\} \\[1ex]
                = & \max\{1,0+1\}=1, \text{ and}\\[1ex]
w\text{-}\dim_v(R_2)=&\max\{w\text{-}\dim_v(T),w\text{-}\dim_v(K(Y))+\dim_v(T_{PT})+d\} \\[1ex]
                  =&\max\{1,0+1+1\}=2.
\end{align*}
This means that $R_2$ is not a $w$-Jaffard domain.

Now set $F:=qf(R_2)$. Suppose that $V:=F+M$ is a rank 1 valuation
domain with maximal ideal $M$. Set $R_3:=R_2+M$. It is easy to see
that $R_3$ is Jaffard of $\dim(R_3)=3$, $w$-$\dim(R_3)=2$, and
$w$-$\dim_v(R_3)=3$. Iterating in the same way we obtain $R_n$
with desired properties.
\end{exam}

Here we give our promised example of a $w$-Jaffard domain which is
not a strong Mori nor a UM$t$ domain.

\begin{exam}\label{nonSM} Let $\mathbb{Q}$ be the field of
rational numbers, $\mathbb{K}=\mathbb{Q}(\sqrt{2},\sqrt{3},\cdots)$
be an algebraic extension field of $\mathbb{Q}$ such that
$[\mathbb{K}:\mathbb{Q}]=\infty$. Let $X$ and $Y$ be indeterminates
over $\mathbb{K}$ and set $R:=\mathbb{Q}+(X,Y)\mathbb{K}[X,Y]$. Then
$R$ is a $w$-Jaffard domain of $w$-dimension 2 using Theorem
\ref{main2}, but it is not a strong Mori domain using \cite[Theorem
3.11]{Mim}. Next we claim that $R$ is not a UM$t$ domain. In fact if
$R$ is a UM$t$ domain, \cite[Corollary 3.2]{FGH} implies that
$(X,Y)$ is a $t$-prime ideal of $\mathbb{K}[X,Y]$, which is absurd
since $\mathbb{K}[X,Y]$ is a Krull domain and $(X,Y)$ has height 2.
Note that in this case $R$ is a 2 dimensional Jaffard domain.
\end{exam}

Recall that an integral domain is called a \emph{Mori domain} if it
satisfies the ascending chain condition on divisorial ideals. Every
strong Mori domain is a Mori domain. The following example is
designed to show that a Mori domain \emph{need not} be a $w$-Jaffard
domain.

\begin{exam} Let $K$ be a field and let $X, Y$ be two
indeterminates over $K$ and set $R:=K+YK(X)[Y]$. Then $R$ is not a
$w$-Jaffard domain by Corollary \ref{g}, but it is a Mori domain by
\cite[Theorem 4.18]{GH}.
\end{exam}

The following example shows that a $w$-Jaffard domain \emph{need
not} be $w$-locally Jaffard.

\begin{exam} Let $K$ be a field and $X_1,X_2$ indeterminates
over $K$. It is proved in \cite[Example 3.2(a)]{ABDFK} that there
are two incomparable valuation domains $(V_1,M_1)$ and $(V_2,M_2)$
of dimension 1 and 2 respectively. Set $T:=V_1\cap V_2$ which is a
two-dimensional Pr\"{u}fer domain with exactly two maximal ideals
$\fm_1$ and $\fm_2$ so that $T_{\fm_1}=V_1$ and $T_{\fm_2}=V_2$.
Denoting $\varphi:T\to T/\fm_1(\cong V_1/M_1\cong K(X_1,X_2))$
consider the pullback $R:=\varphi^{-1}(K(X_1))$. Since $K(X_1)$ and
$T$ are DW-domains  , \cite[Theorem 3.1(3)]{Mim1} implies that $R$
is also a DW-domain. In particular $w$-$\dim(R)=\dim(R)$ and
$w$-$\dim_v(R)=\dim_v(R)$. It follows that
$w$-$\dim(R)=\max\{2,0+1\}=2$, $w$-$\dim_v(R)=\max\{2,0+1+1\}=2$.
Thus $R$ is a $w$-Jaffard ($=$Jaffard) domain. It is observed in
\cite[Example 3.2(a)]{ABDFK} that for the prime ideal
$\fn_1:=\fm_1\cap R$ of $R$, $\dim(R_{\fn_1})=1$ and
$\dim_v(R_{\fn_1})=2$. This shows that $R$ is not $w$-locally
Jaffard.
\end{exam}

By \cite[Exercise 17(1), Page 372]{G}, for each positive integer
$n$, there exists a finite-dimensional (non-Jaffard) domain $R$
such that $\dim_v(R)-\dim(R)=n$ (see also \cite[Example
3.1(a)]{ABDFK}).

\begin{exam} For each positive integer
$n$, there exists a finite $w$-dimensional domain $R$ such that
$w\text{-}\dim_v(R)-w\text{-}\dim(R)=n$.

\noindent Indeed let $D$ be a Krull domain. Let $K$ be the
quotient field of $D$ and $\{X_1,\cdots,X_n,Y\}$ be a set of $n+1$
indeterminates over $K$. Let $L$ denote the field
$K(X_1,\cdots,X_n)$. Also define the valuation domain
$V:=L[Y]_{(Y)}=L+M$ (with $M=YV$) and the ring $R:=D+M$. Applying
Propositions \ref{wdim=dim} and \ref{main} to the pullback
description of $R$, we have $w\text{-}\dim(R)=w\text{-}\dim(D)+1$
and $w\text{-}\dim_v(R)=w\text{-}\dim_v(D)+1+n$. Since $D$ is a
Krull domain we have $w\text{-}\dim(D)=w\text{-}\dim_v(D)=1$. So
that $w\text{-}\dim_v(R)-w\text{-}\dim(R)=n$. In particular $R$
is not a $w$-Jaffard domain. Note that $\dim(R)=\dim(D)+1$ by
\cite[Proposition 2.1(5)]{Fon} and $\dim_v(R)=\dim_v(D)+1+n$ by
\cite[Theorem 2.6(a)]{ABDFK} while $w\text{-}\dim(R)=2$ and
$w\text{-}\dim_v(R)=2+n$. In particular if $\dim(D)\geq2$ then
$\dim(R)\neq w\text{-}\dim(R)$ and $\dim_v(R)\neq
w\text{-}\dim_v(R)$.
\end{exam}

We remark that in the above example if $D$ is a Dedekind domain
then $R$ is a DW-domain by \cite[Theorem 3.1(2)]{Mim1}. This
means that $\dim(R)=w\text{-}\dim(R)$ and
$\dim_v(R)=w\text{-}\dim_v(R)$. Therefore $R$ is a non-Jaffard
domain with $\dim_v(R)-\dim(R)=n$.

\section{An application}

Recall that in \cite{S2} Seidenberg proved that if $n,m$ are
positive integers such that $n+1\leq m\leq 2n+1$, there is an
integrally closed domain $R$ such that $\dim(R)=n$ and
$\dim(R[X])=m$. More recently in \cite[Theorem 2.10]{F2} Wang
showed that for any pair of positive integers $n,m$ with $1\leq
n\leq m\leq 2n$, there is an integrally closed domain $R$ such
that $w$-$\dim(R)=n$ and $w$-$\dim(R[X])=m$. By Proposition
\ref{finite} we have for an integral domain $R$ if
$n=w$-$\dim(R)$ then
$$n+1\leq w[X]\text{-}\dim(R[X])\leq 2n+1.$$
Now we show that these bounds are the best possible. We say that an
integral domain $R$ is of \emph{$w_x$-type $(n,m)$} if
$w$-$\dim(R)=n$ and $w[X]$-$\dim(R[X])=m$.

\begin{thm}\label{main4} Let $D$ be an integral domain of $w_x$-type
$(n,m)$ with quotient field $K$. Let $L$ be a purely
transcendental field extension of $K$. Then:
\begin{itemize}
\item[(a)] If $V_1=K+M_1$ is a DVR and $R_1=D+M_1$, then $R_1$ is
of $w_x$-type $(n+1,m+1)$.

\item[(b)] If $V_2=L+M_2$ is a DVR and $R_2=D+M_2$, then $R_2$ is
of $w_x$-type $(n+1,m+2)$.
\end{itemize}
\end{thm}

\begin{proof} (a) Using Proposition \ref{wdim=dim} we have
$$
w\text{-}\dim(R_1)=w\text{-}\dim(D)+\dim(V_1)=n+1.
$$
Since $R_1$ is a pullback of a diagram of type $(\square^*)$,
Proposition \ref{main} yields that
$$
w[X]\text{-}\dim(R_1[X])=w[X]\text{-}\dim(D[X])+\dim(V_1[X])-1=m+2-1=m+1.
$$

(b) By the same way as (a) we have $w\text{-}\dim(R_2)=n+1$. Now
we compute $w[X]\text{-}\dim(R_2[X])$. If $Q_1\subset\cdots\subset
Q_m$ is a chain of $w[X]$-prime ideals of $D[X]$ of length $m$,
then
$$
M_2[X]\subset Q_1+M_2[X]\subset\cdots\subset Q_m+M_2[X]
$$
is a chain of prime ideals of $R_2[X]$ of length $m+1$. Notice
that $(Q_i+M_2[X])\cap R_2=Q_i\cap D+M_2$ for $i=1,\cdots,m$.
Since $Q_i$ is a $w[X]$-prime ideal of $D$ then $Q_i\cap D$ is a
$w$-prime ideal of $D$ (or equal to zero) by \cite[Remark 2.3]{S}.
Therefore by \cite[Lemma 2.3]{F2} we see that $Q_i\cap
D+M_2=(Q_i+M_2[X])\cap R_2$ is a $w$-prime ideal of $R_2$. Thus
using \cite[Remark 2.3]{S} we obtain that $Q_i+M_2[X]$ is a
$w[X]$-prime ideal of $R_2[X]$. On the other hand
$(R_2)_{M_2}=K+M_2$, and therefore it is not a valuation domain.
Thus \cite[Theorem 19.15(2)]{G} yields that $M_2[X]$ is not
minimal in $R_2[X]$. Therefore $w[X]\text{-}\dim(R_2[X])\geq
m+2$. We consider a chain $P_1\subset\cdots\subset P_s$ of
$w[X]$-prime ideals of $R_2[X]$ of maximal length. Since $P_2$ is
not minimal in $R_2[X]$, $P_2\cap R_2\neq(0)$. By \cite[Part (3)
of Exercise 12 Page 202]{G} $M_2$ is the unique minimal prime
ideal of $R_2$. Therefore $M_2\subseteq P_2\cap R_2$ and
$M_2[X]\subseteq P_2$. Each $P_j\cap R_2$ is a $w$-prime ideal of
$R_2$ by \cite[Remark 2.3]{S} for $j=1,\cdots,s$. Since
$(P_j/M_2[X])\cap D=(P_j/M_2[X])\cap R_2/M_2=(P_j\cap R_2)/M_2$
we claim that $(P_j\cap R_2)/M_2=(P_j/M_2[X])\cap D$ is a
$w$-prime ideal of $D$ by Lemma \ref{w}. Therefore $P_j/M_2[X]$
is a $w[X]$-prime ideal of $D[X]$ by \cite[Remark 2.3]{S}. So that
$P_3/M_2[X]\subset\cdots\subset P_s/M_2[X]$ is a chain of
$w[X]$-prime ideals of $D[X]$, and thus $s-2\leq m$. It follows
that $w[X]\text{-}\dim(R_2[X])=m+2$ completing the proof.
\end{proof}

\begin{rem} Let $D^c$ and $R_i^c$ denote the integral closures of $D$ and $R_i$ in
their quotient fields, respectively ($i=1,2$). Then $R_i^c=D^c+M_i$
by \cite[Lemma 2.6(1)]{F2}. Therefore, $R_i$ is integrally closed if
and only if $D$ is integrally closed.
\end{rem}

Following Seidenberg, we say that a domain $R$ is an \emph{F-ring}
if $\dim(R)=1$ and $\dim(R[X])=3$. By \cite[Corollary 3.6]{HZ} and
\cite[Proposition 30.14]{G}, a one-dimensional domain $R$ is an
F-ring if and only if $R$ is not a UM$t$-domain. For an F-ring,
$w$-$\dim(R)=1$ and $w[X]$-$\dim(R[X])=3$ by \cite[Corollary
3.6]{S}. Thus a F-ring is a domain of $w_x$-type $(1,3)$.

\begin{cor} For any pair of positive integers $(n,m)$ with
$n+1\leq m\leq 2n+1$, there is an integrally closed integral
domain $R$ of $w_x$-type $(n,m)$.
\end{cor}

\begin{proof} A PID is an integrally closed integral
domain of $w_x$-type $(1,2)$. By \cite[Theorem 3]{S2} there is an
integrally closed F-ring. Thus by the comments before the corollary
it is of $w_x$-type $(1,3)$. So if $n=1$ the result is true. Using
Theorem \ref{main4} and by an induction argument similar to the
proof of \cite[Theorem 3]{S2}, the proof is complete.
\end{proof}




\begin{thebibliography}{10}

\bibitem{ABDFK} D. F. Anderson, A. Bouvier, D. Dobbs, M. Fontana and S. Kabbaj, {\em On Jaffard domain},
Expo. Math., {\bf 6}, (1988), 145--175.

\bibitem{A} J. T. Arnold, {\em On the dimension theory of overrings of an integral domain},
Trans. Amer. Math. Soc., {\bf 138}, (1969), 313--326.

\bibitem{AG} J. T. Arnold and R. Gilmer, {\em The dimension sequence of a commutative ring},
Amer. J. Math., {\bf 96}, (1974), 385--408.

\bibitem{BG} E. Bastida and R. Gilmer, {\em Overrings and divisorial ideals
of rings of the form $D+M$}, Michigan Math. J. {\bf 20}, (1973),
79--95.

\bibitem{BSH} M. Boisen and P. Sheldon, {\em CPI-extensions: overrings of integral
domains with special prime spectrum}, Canad. J. Math. {\bf 29},
(1977), 722--737.

\bibitem{BK} A. Bouvier and S. Kabbaj, {\em Examples of Jaffard domains}, J. Pure Appl. Algebra,
{\bf 54}, (2-3), (1988), 155--165.

\bibitem{Fon} M. Fontana, {\em Topologically defined classes of commutative rings}, Ann. Mat. Pura Appl. {\bf
123}, (1980), 331--355.

\bibitem{FGH} M. Fontana, S. Gabelli and E. Houston, {\em UMT-domains and domains
with Pr\"{u}fer integral closure}, Comm. Algebra {\bf 26}, (1998),
1017--1039.

\bibitem{FHP} M. Fontana, J. Huckaba, and I. Papick, {\em Pr\"{u}fer domains}, New York, Marcel
Dekker, 1997.

\bibitem{FL} M. Fontana and K. A. Loper, {\em Nagata rings, Kronecker function
rings and related semistar operations}, Comm. Algebra {\bf 31}
(2003), 4775--4801.

\bibitem{GH} S. Gabelli and E. Houston, {\em Coherentlike conditions in
pullbacks}, Michigan Math. J. {\bf 44}, (1997), 99--123.

\bibitem{GH1} S. Gabelli and E. Houston, {\em Ideal theory in
pullbacks} Non-Noetherian commutative ring theory, 199--227, Math.
Appl., 520, Kluwer Acad. Publ., Dordrecht, 2000.

\bibitem{G} R. Gilmer, {\em Multiplicative ideal theory}, New York, Dekker, 1972.

\bibitem{H} E. Houston, {\em Prime $t$-ideals in $R[X ]$}, in: P.-J. Cahen, D.G. Costa, M. Fontana, S-E. Kabbaj (Eds.),
Commutative Ring Theory, Dekker Lecture Notes {\bf 153}, 1994,
163--170.

\bibitem{HM} E. Houston and A. Mimouni, {\em
On the divisorial spectrum of a Noetherian domain}, J. Pure Appl.
Algebra {\bf 214} (2010), 47--52.

\bibitem{HZ} E. Houston and M. Zafrullah, {\em On $t$-invertibility, II}, Comm. Algebra
{\bf 17}, (1989), 1955--1969.

\bibitem{Jaf1} P. Jaffard, {\em Th\'{e}orie de la dimension dans les anneaux de polynomes}, Gauthier-Villars,
Paris, 1960.

\bibitem{Kr} W. Krull, {\em Jacobsonsche Rings, Hilbertsche Nullstellensatz, Dimension Theorie},
Math. Z. {\bf 54}, (1951), 354--387.

\bibitem{Na} M. Nagata, {\em Local rings}, Wiley-Interscience, New York, 1962.

\bibitem{Mim} A. Mimouni, {\em TW-domains and strong Mori domains}, J. Pure Appl. Algebra
{\bf 177}, (2003), 79--93.

\bibitem{Mim1} A. Mimouni, {\em Integral domains in which each ideal is a $w$-ideal}, Comm. Algebra
{\bf 33}, No. 5, (2005), 1345--1355.

\bibitem{S} P. Sahandi, {\em Semistar-Krull and valuative dimension of integral
domains}, Ricerche Mat., {\bf 58}, (2009), 219--242.

\bibitem{Sah} P. Sahandi, {\em Universally catenarian integral domains, strong S-domains and semistar operations},
Comm. Algebra {\bf 38}, No. 2, (2010), 673--683.

\bibitem{Sbull} P. Sahandi, {\em Semistar dimension of polynomial rings and Pr\"{u}fer-like
domains}, Bull. Iranian Math. Soc., to appear, arXiv:0808.1331v2
[math.AC].

\bibitem{Sq} P. Sahandi, {\em On quasi-Pr\"{u}fer and UM$t$ domains},
preprint (2010).

\bibitem{S1} A. Seidenberg, {\em A note on the dimension theory of rings}, Pasific J. Math.
{\bf 3}, (1953), 505--512.

\bibitem{S2} A. Seidenberg, {\em On the dimension theory of rings II}, Pasific J. Math.
{\bf 4}, (1954), 603--614.

\bibitem{F1} F. G. Wang, {\em On $w$-dimension of domains},
Comm. Algebra {\bf 27}, (1999), 2267--2276.

\bibitem{F2} F. G. Wang, {\em On $w$-dimension of domains, II}, Comm. Algebra {\bf
29}, (2001), 2419--2428.

\bibitem{WM} F. G. Wang and R. L. McCasland, {\em On w-modules over strong Mori
Domains}, Comm. Algebra {\bf 25}, (1997), 1285--1306.

\end{thebibliography}
\end{document}